\documentclass[10pt]{article}

\usepackage{amsfonts}
\usepackage{amssymb}
\usepackage{amsmath,stmaryrd,hyperref}
\usepackage{amsmath}
\usepackage{amsfonts}
\usepackage{amssymb}
\usepackage{ntheorem}
\usepackage{mathrsfs}
\usepackage{stmaryrd}
\usepackage{bussproofs}
\usepackage{pdfsync}
\usepackage{upgreek}
\usepackage{cite}
\usepackage{proof,refcount,xstring,appendix}

\usepackage{tikz,pgf}
\usetikzlibrary{automata,arrows,shapes,decorations,topaths,trees,backgrounds,shadows}
\usepgfmodule{matrix}

{\end{list}}

\def\N{\mathcal{N}}

\def\M{\mathfrak{M}}

\newcommand\axlabel[2][\relax]{
\ifx\relax#1\tag{#2} \else \tag{#1}\fi
\label{#2}}\newcommand\axref[1]{(\ref{#1})}

\def\inte{B1}
\def\intz{B2}
\def\intv{B3}
\def\ints{B4}

\def\INTe{\axref{\inte}}
\def\INTz{\axref{\intz}}
\def\INTv{\axref{\intv}}
\def\INTs{\axref{\ints}}

\newcommand{\sys}[1]{\ensuremath{\mathbf{#1}}}
\newcommand{\verberg}[1]{}
\newcommand{\old}[1]{}

\newenvironment{proof}{\noindent \emph{Proof.}}
 { \rule{0.5em}{0.5em}}
 { \rule{0.5em}{0.5em}}

  \newcounter{memory}
   \newcounter{sectmemory}

\newenvironment{rtheorem}[1]{
\setcounter{memory}{\value{thm}}   
\setcounter{sectmemory}{\value{section}}   
\StrBehind{\getrefnumber{#1}}{.}[\collectnumber]   
\setcounter{thm}{0\collectnumber}                          
\StrBefore{\getrefnumber{#1}}{.}[\collectsectionnumber]   
\setcounter{section}{0\collectsectionnumber}
\addtocounter{thm}{-1}                                             
\begin{theorem}}
{\end{theorem}
\setcounter{section}{\value{sectmemory}} 
\setcounter{thm}{\value{memory}
}}
\newtheorem{thm}{Theorem}

\newtheorem{lemma}[thm]{Lemma}

\newtheorem{theorem}{Theorem}
\newtheorem{definition}{Definition}
\newtheorem{corollary}{Corollary}

\newenvironment{rlemma}[1]{
\setcounter{memory}{\value{thm}}   
\setcounter{sectmemory}{\value{section}}   
\StrBehind{\getrefnumber{#1}}{.}[\collectnumber]   
\setcounter{thm}{0\collectnumber}                          
\StrBefore{\getrefnumber{#1}}{.}[\collectsectionnumber]   
\setcounter{section}{0\collectsectionnumber}
\addtocounter{thm}{-1}                                             
\begin{lemma}}
{\end{lemma}
\setcounter{section}{\value{sectmemory}} 
\setcounter{thm}{\value{memory}
}}

\begin{document}
  \title{Pointwise intersection in neighbourhood modal logic}
  \date{}
  \author{Frederik Van De Putte\thanks{Post-doctoral fellow of the Flemish Research Foundation -- FWO Vlaanderen. We are indebted to Olivier Roy and Eric Pacuit for comments on preparatory notes for this paper.} ~and Dominik Klein\thanks{The work of DK  was partially supported by the Deutsche
Forschungsgemeinschaft (DFG) and  Agence Nationale de la Recherche (ANR) as part of the joint project Collective Attitude Formation [RO 4548/8-1] and by DFG and
Grantov\'a Agentura \v{C}esk\'e Republiky (GA\v{C}R) as part of the joint project From Shared Evidence to Group Attitudes [RO 4548/6-1].}}

\maketitle

  \begin{abstract}
 We study the logic of neighbourhood models with \emph{pointwise intersection}, as a means to characterize multi-modal logics. Pointwise intersection takes us from a set of neighbourhood sets $\mathcal{N}_i$ (one for each member $i$ of a set $G$, used to interpret the modality $\square_i$) to a new neighbourhood set $\mathcal{N}_G$, which in turn allows us to interpret the operator $\square_G$. Here, $X$ is in the neighbourhood for $G$ if and only if $X$ equals the intersection of some $\mathcal{Y} = \{Y_i \mid i\in G\}$. We show that the notion of pointwise intersection has various applications in epistemic and doxastic logic, deontic logic, coalition logic, and evidence logic. We then establish sound and strongly complete axiomatizations for the weakest logic characterized by pointwise intersection and for a number of variants, using a new and generally applicable technique for canonical model construction.  \vspace{14pt}

 
\noindent \end{abstract}

\section{Introduction}

Neighbourhood semantics is a well-established tool to study generalizations and variants of Kripke-semantics for modal logic.\footnote{Scott \cite{Scott1970} and Montague \cite{Montague:1970} are often seen as the inventors of neighbourhood models; Chellas \cite{Chellas:ml} and Segerberg \cite{Segerberg:1971} are usually cited as the main figures in their development.} They have been successfully applied to i.a.\ the logic of ability \cite{Brown:logicab,Pauly:2002}, the dynamics of evidence and beliefs \cite{vanBenthem&Pacuit:2011}, conflict-tolerant deontic logic \cite{LG:ldd}, and the analysis of (descriptive or normative) conditionals \cite{DL:c,C&J2002}.

Formally, a neighbourhood function $\mathcal{N}:W\rightarrow\wp(\wp(W))$ yields a set of accessible sets $X_1,X_2,\ldots$ of worlds for every given world $w$ in a possible worlds model. $\square\varphi$ is then true iff there is some such $X$ in the neighbourhood set $\mathcal{N}(w)$, that coincides with the truth set of $\varphi$ (cf.\ Definitions \ref{def:model} and \ref{def:semclauses} below).

The move from Kripke semantics to neighbourhood semantics allows us to invalidate certain schemata that are problematic for a given interpretation of the modal operator $\square$, but also to include other schemata that would trivialize any normal modal logic.\footnote{See Table \ref{table:cheapcond} in Section \ref{sec:extcheap} for examples.} Apart from that, neighbourhood models can also be used as a purely technical vehicle in order to arrive at completeness or incompleteness w.r.t.\ less abstract possible worlds semantics.\footnote{One prototypical example of a completeness proof via neighbourhood semantics is \cite{DL:c}. In \cite{Governatori:2005}, neighbourhood semantics are used to prove the incompleteness of Elgesem's modal logic of agency \cite{Elgesem:1997}. We refer to \cite{Pacuit:2017} for a critical introduction to the many forms, uses and advantages of neighbourhood semantics.}

Many applications in philosophy and AI require a multitude of modal operators $\square_1,\square_2,\ldots$, where the indices may represent agents (logic of agency, doxastic or epistemic logic), non-logical axioms or reasons (logic of provability or normative reasoning), or sources of a norm (deontic logic) or of evidence (doxastic logic once more). Just as for Kripke-semantics, the step from the setting with only one modal operator to a multi-indexed one is easily made, as long as no interaction among the various operators, resp.\ neighbourhood functions is presupposed. However, the logic of neighbourhood models where certain neighbourhood functions are obtained by operations on (one or several) other neighbourhood functions is still largely unknown. This stands in sharp contrast to the current situation in Kripke-semantics, cf.\ the abundant literature on Dynamic Logic \cite{Hareletal:2000} and on Boolean Modal Logic \cite{GargovPassy:booleanmodallogic,Gargov-booleanspec}.

The current paper is a first step towards filling this gap. In particular, we study logics that are interpreted in terms of the \emph{pointwise intersection} of neighbourhoods. This concept is defined as follows, for a fixed (finite or infinite) index set $I = \{1,2,\ldots\}$ and a fixed set of atomic propositions $\mathfrak{P}$. 

\begin{definition}\label{def:model} A \emph{model}  $\mathfrak{M}$ is a triple $\langle W,\langle \mathcal{N}_i\rangle_{i\in I}, V\rangle$, where $W\neq\emptyset$ is the \emph{domain} of $\mathfrak{M}$, for every $i\in I$, $\mathcal{N}_i: W\rightarrow \wp(\wp(W))$ is a \emph{neighbourhood function for $i$}, and $V: \mathfrak{P}\rightarrow \wp(W)$ is a \emph{valuation function}.

Where  $\mathfrak{M}= \langle W,\langle \mathcal{N}_i\rangle_{i\in I}, V\rangle$ is a model and $G = \{i_1,\ldots,i_n\}\subseteq I$, the \emph{neighbourhood function for $G$} is given by

$$\mathcal{N}_G(w) = \{X_{i_1}\cap \ldots\cap X_{i_n} \mid  \mbox{ each } X_{i_j} \in \mathcal{N}_{i_j}(w)\}$$

\end{definition}

So, in the context of neighbourhood semantics, pointwise intersection takes as input any intersection of neighbourhoods, one for each agent $i\in G$, to form the new neighbourhood set for $G$. This new neighbourhood set is then used to interpret expressions of the type $\square_G\varphi$, by means of the standard semantic clause, plugging in the neighbourhood function $\mathcal{N}_G$.\footnote{One obvious question, especially if we do not assume that the neighbourhood sets $\mathcal{N}_i(w)$ are closed under intersection, is whether we can also have pointwise intersection of a neighbourhood set with itself. The short answer is: yes, we can, but this takes us beyond the scope of this conference paper. We return to this point in our concluding section.}

Beside its mathematical interest, pointwise intersection has many potential applications. In Section \ref{sec:app}, we briefly point out a few of these. Sections \ref{sec:bl}--\ref{sec:comp} form the technical core of the paper, providing (strong) soundness and completeness results for a number of logics interpreted in terms of models with pointwise intersection. We conclude with a summary and some open questions for future work.

\section{Applications}\label{sec:app}

What follows is a non-exhaustive list of (potential) applications of logics with pointwise intersection. We leave the full elaboration of these ideas for later occasions, and whenever possible, provide pointers to the literature for more background information. 

\paragraph{Epistemic and Doxastic Logic} The \emph{distributed knowledge} of a group of agents $G$ can be conceived as the knowledge that would be obtained if \old{all members of $G$ would put together what they each know individually, and close the result under logical consequence.}{some third agent combined the individual knowledge of all group members $G$ and closed the result under logical consequence \cite{AagotnesDistBel}.} 
The logic of this notion is then defined as an extension of a multi-agent version of \sys{S5}, where each operator $\square_G$ ($G\subseteq I$) is interpreted in terms of the intersection of the equivalence relations $R_i$ ($i\in G$).\footnote{The logic of distributed knowledge is investigated in the seminal work \cite{faginetal:RAK}. A small warning is in place here though. As Gerbrandy \cite{Gerbrandky:DK} shows, the notion of distributed knowledge has both a syntactic and a semantic reading, which are not entirely equivalent. Fagin and co-authors \cite{faginetal:RAK}, and most others in the field focus on the semantically driven view.
} Analogously, one can study \emph{distributed beliefs} of a group $G$ as the result of aggregating (or pooling) all the beliefs of the members of $G$. Formally, distributed belief can be seen as all combinations of pieces of belief, one for each agent. When beliefs are conceived as neighbourhoods, the operation of pooling one's beliefs corresponds to a pointwise intersection.

In his \cite{StalnakerKB} Robert Stalnaker has proposed a combined epistemic-doxastic logic that interprets belief as the mental component of knowledge. In the framework, he abandons the assumption that knowledge is negatively introspective. Also positive introspection has been heavily critisized on philosophical grounds. Correspondingly, \cite{KRG} propose two logics that weaken Stalnaker's framework further by also omitting positive introspection. It turns out that this renders belief a non-normal modality:  belief is closed under weakening but not under intersection, i.e.\ the agent can believe $\varphi$ and $\psi$ without believing $\varphi\wedge\psi$. 
This is but one example of a non-normal logic for knowledge and belief. {All such logics raise the question of defining group attitudes for non-normal modal logics akin to the distributed knowledge and belief defined above.  } 
 Our results show that group versions of non-normal knowledge and belief can be easily axiomatized.
To use a slogan: we can throw away the normal modal logic bathwater, while keeping the distributed knowledge/belief baby.

\paragraph{Evidence Logic} The framework of Evidence Logic was proposed in \cite{vanBenthem&Pacuit:2011} to study the way beliefs (of a given agent) are grounded in (possibly conflicting) evidence. Technically, evidence logics are obtained  by adding a monotonic operator\footnote{A modal operator $\square$ is \emph{monotonic} in a given system iff it satisfies the rule: from $\varphi\vdash \psi$, to infer $\square\varphi\vdash \square\psi$.} ${\sf E}$ for ``the agent has evidence for ...'' and a belief operator ${\sf B}$ of the type $\mathbf{KD45}$ to classical logic. ${\sf E}$ is characterized semantically in terms of a neighbourhood function $\mathcal{N}$, where $X\in \mathcal{N}(w)$ expresses that at $w$, the agent has evidence for $X$. The belief state at a world $w$ is interpreted as the union of all intersections $\bigcap\mathcal{X}$, where $\mathcal{X}$ is a maximal set of evidence such that $\bigcap\mathcal{X}\neq\emptyset$.

Going multi-agent with this framework is fairly straightforward. Here, our results can e.g.\ be used to study the piecemeal aggregation of evidence from various different sources, and how diverging strategies to do so impact the resulting belief set. Formally, $X\in\mathcal{N}_{\{i,j,k\}}(w)$ indicates that $X$ is a result of aggregating pieces of evidence of the sources $i$, $j$, and $k$. One interesting epistemological question -- that can now be studied at a logical level -- is whether it makes a difference if one first aggregates the evidence among the sources, before computing a set of beliefs, rather than using the evidence in its original form (ignoring the sources) to ground the beliefs.


\paragraph{Deontic Logic} Neighbourhood semantics have been used in Deontic Logic to model (non-explosive) conflict-tolerant normative reasoning \cite{LG:ldd,Goblehandbook}. Here, $\square_i \varphi$ can e.g.\ be used to express that there is at least one norm in the normative system $S_i$ that makes $\varphi$ obligatory; the presence of two conflicting norms in $S_i$ can then account for the truth of a deontic conflict of the type $\square_i \varphi\wedge\square_i \neg\varphi$. In this context, pointwise intersection can be interpreted as the piecemeal aggregation of norms from different normative systems; a formula such as $\square_{\{1,2\}} p$ then expresses that there are two norms, one in $S_1$, the other in $S_2$, such that obeying both norms entails that $p$ is the case.


An altogether different application of the formal framework developed here consists in reading the indices as \emph{reasons} for one's obligations. On this view, $\square_r \varphi$ expresseses that $r$ is a reason for $\varphi$ to be obligatory, and one can then aggregate reasons alongside with obligations: $\square_r \varphi \wedge \square_{r'}\psi$ yields $\square_{\{r,r'\}}(\varphi\wedge\psi)$. As argued in \cite{NairHorty:fc,OUR}, reasons play an important, but often neglected role in our normative reasoning; a thorough logical investigation of their interaction and aggregation in deontic logic is still largely lacking.

\paragraph{Coalition Logic, group abilities} As shown in \cite{Broersenetal:nsCL}, Pauly's Coalition Logic \cite{Pauly:2002} corresponds to the ability-fragment of \textbf{STIT} logic \cite{Belnapetal:ff,Horty:adl}. Moreover, this fragment is known to be decidable, in contrast to full \sys{STIT} logic for groups \cite{HerzigSchwartzentruber:prop}. In Coalition Logic, $\square_{G}\varphi$ expresses that ``the group of agents $G$ has the ability to ensure that $\varphi$ is the case'', or in more game-theoretic terminology, ``$G$ is $\alpha$-effective for $\varphi$''. The modality $\square_G$ is monotonic, meaning that we can only express what one of the group's choices \emph{necessitates} -- not what defines that choice.\footnote{Technically, $\square_G$ is monotonic iff it satisfies the rule (RM): if $\square\varphi$ and $\varphi\vdash\psi$, then $\square\psi$. We discuss this rule in Section \ref{subsec:closuresubsets}.} With the results of the current paper, we can now also obtain sound and complete logics for \emph{exact} ability (resp.\ effectivity), where $\square_G \varphi$ means that ``$G$ can make a choice that is defined by $\varphi$'', or in more mundane terms: ``$G$ can do exactly $\varphi$''.


\section{The Base Logic}\label{sec:bl}

In the remainder we use $\mathfrak{M}$, $\mathfrak{M}'$ to refer to arbitrary models as given by Definition \ref{def:model}. $X,Y,\ldots$ are used to refer to sets of worlds in a model, and $w,w',\ldots$ for single worlds. We write $G\subseteq_f I$ to denote that $G$ is a finite subset of $I$.

Let $\mathfrak{L}$ be the language obtained by closing a countable set of propositional variables $\mathfrak{P} = \{p,q,\ldots\}$ and the logical constants $\bot,\top$ under the classical connectives and all unary modal operators of the type $\square_G$, where $G\subseteq_f I$. We use $\varphi,\psi,\ldots$ as metavariables for formulas and $\Gamma,\Delta,\ldots$ as metavariables for sets of formulas. To interpret $\mathfrak{L}$, we use the models given by Definition \ref{def:model} together with the following (standard) semantic clauses:\footnote{We treat $\bot,\neg,\vee$ as primitive; the other connectives and $\top$ are defined in the standard way.}

\begin{definition}\label{def:semclauses} Where  $\mathfrak{M}= \langle W,\langle \mathcal{N}_i\rangle_{i\in I}, V\rangle$ is a model, $w\in W$, $\varphi,\psi\in\mathfrak{L}$, and $G\subseteq_f I$:
\begin{itemize}
\item[0.] $\mathfrak{M},w\not\models\bot$
\item[1.] $\mathfrak{M},w \models \varphi$ iff $w\in V(\varphi)$ for all $\varphi\in\mathfrak{P}$
\item[2.] $\mathfrak{M},w\models \neg\varphi$ iff $\mathfrak{M},w\not\models \varphi$
\item[3.] $\mathfrak{M},w\models \varphi\vee\psi$ iff $\mathfrak{M},w\models \varphi$ or $\mathfrak{M},w\models \psi$
\item[4.] $\mathfrak{M},w\models \square_G\varphi$ iff $\|\varphi\|^{\mathfrak{M}} \in \mathcal{N}_G(w)$
\end{itemize}
\noindent where $\|\varphi\|^{\mathfrak{M}} = \{w\in W\mid \mathfrak{M},w\models \varphi\}$.
\end{definition}

Validity($\Vdash \varphi$) and semantic consequence ($\Gamma\Vdash \varphi$), for a given class of models, are defined in the standard way, viz.\ as truth, resp.\ truth-preservation at all worlds in all models in that class. 

In the remainder of this paper, we will consider various logics that are obtained by imposing certain frame conditions on the models defined above. We start with the \emph{base logic}, i.e.\ the logic characterized by the class of all models. To characterize this logic syntactically, we will need the following axioms in addition to classical logic (henceforth, \sys{CL}):

\begin{align*}
&\mbox{where } G\cap H = \emptyset: (\square_G\varphi\wedge\square_H\psi) \rightarrow \square_{G\cup H} (\varphi\wedge\psi) \axlabel{\inte}\\[.3em]
&\square_{G\cup H}\top \rightarrow \square_G \top \axlabel{\intz}\\[.3em]
&(\square_G \varphi \wedge \square_{G\cup H\cup J}\varphi) \rightarrow \square_{G\cup H} \varphi \axlabel{\intv}\\[.3em]
&(\square_G\varphi\wedge \square_H (\varphi \vee\psi)) \rightarrow \square_{G\cup H}\varphi \axlabel{\ints}
\end{align*}

\verberg{

\begin{itemize}
\item[] where $G\cap H = \emptyset$: $(\square_G\varphi\wedge\square_H\psi) \rightarrow \square_{G\cup H} (\varphi\wedge\psi)$ \hfill{\INTe}
\item[] $\square_{G\cup H}\top \rightarrow \square_G \top$ \hfill{\INTz}
\item[] $(\square_G \varphi \wedge \square_{G\cup H\cup J}\varphi) \rightarrow \square_{G\cup H} \varphi$ \hfill{\INTv}
\item[] $(\square_G\varphi\wedge \square_H (\varphi \vee\psi)) \rightarrow \square_{G\cup H}\varphi$ \hfill{\INTs}
\end{itemize}
}

\noindent and, as usual, replacement of equivalents and modus ponens:

\verberg{
\begin{itemize}
\item[] if $\varphi\vdash\psi$ and $\psi\vdash\varphi$, then $\square_G \varphi\vdash\square_G\psi$ \hfill{(RE)}
\item[] if $\vdash\varphi$ and $\vdash\varphi\rightarrow\psi$, then $\vdash\psi$ \hfill{(MP)}
\end{itemize}
}
\begin{align*}
&\text{if }\varphi\vdash\psi\text{ and }\psi\vdash\varphi\text{, then }\square_G \varphi\vdash\square_G\psi \axlabel{RE}\\[.3em]
&\text{if }\vdash\varphi\text{ and }\vdash\varphi\rightarrow\psi\text{, then }\vdash\psi \axlabel{MP}
\end{align*}

Let us quickly offer some interpretations of these axioms. \INTe\, is an obvious syntactic consequence of taking intersections: If $\|\varphi\|^{\mathfrak{M}}$ is in $G$'s neighbourhood and $\|\psi\|^{\mathfrak{M}}$ is in $H$'s neighbourhood, then $\|\varphi\wedge\psi\|^{\mathfrak{M}}$ is in their intersection neighbourhood whenever $G$ and $H$ are disjoint. Note that the latter restriction is required; without it, the axiom is not sound for the base logic.\footnote{To see why, note that neighbourhood functions are not generally assumed to be closed under intersection: $X,Y\in\mathcal{N}_i(w)$ does not imply $X\cap Y\in\mathcal{N}(w)$. The unrestricted version of \INTe\, includes the case where $G=H=\{i\}$, which is only sound if neighbourhoods are closed under intersection. We return to this point in Section \ref{subsec:intersect}.} Axiom \INTz\, states that $W$ can  only be in $G$'s intersection neighbourhood if it is in the neighbourhood of each member of $G$.
\INTv\, expresses a property of convex closure: if $X\in \mathcal{N}_G(w)$ and $X\in\mathcal{N}_{G\cup H\cup J}(w)$, then for all $i\in H$, there must be a $Y_i\in \mathcal{N}_i(w)$ such that $X \subseteq Y_i$. Consequentially, also $X\in\mathcal{N}_{G\cup H}(w)$. \INTs\, follows the same reasoning as \INTv\, but is logically independent. {In the appendix we prove the following:}

\begin{lemma}\label{Indeplem}
Axioms \INTe-\INTs\ are logically independent from each other.
\end{lemma}

Before we move to the completeness proof, some terminological remarks are needed. In this and the next section, we use Hilbert-style axiomatizations, with \axref{MP} and \axref{RE} as our only rules. We work with axiom schemata; an axiom is any instance of an axiom schema in $\mathfrak{L}$. Every formula in $\mathfrak{L}$ that can be derived by the axioms and rules is a theorem of the logic. Finally, consequence relations are defined from the respective axiomatizations as follows: $\Gamma\vdash\varphi$ iff there are $\psi_1,\ldots,\psi_n\in\Gamma$ such that $(\psi_1\wedge \ldots\wedge\psi_n)\rightarrow\varphi$ is a theorem. Note that this means that the syntactic consequence relation of the defined logics is by definition compact.




\section{Strong Completeness for the Base Logic}

In this section, we prove the following:

\begin{theorem}[Strong Completeness for the Base Logic]\label{basethm}
A sound and strongly complete axiomatization of the base logic is obtained by adding \INTe, \INTz, \INTv, and \INTs\, to any sound and complete axiomatization of \sys{CL}, and closing the result under \axref{RE} and \axref{MP}.
\end{theorem}

The proof of soundness is a matter of routine; it suffices to check that all the axioms are sound with respect to the class of all neighbourhood models. 
For the completeness proof, we need to construct a canonical model $\mathfrak{M}^c$, in which every world corresponds to a maximal consistent set (MCS) of formulas $\Lambda\subseteq\mathfrak{L}$.  The main difficulty here is to construct the $\mathcal{N}_i$ in such a way that (a) if a given formula $\square_G\varphi$ has to be true at a world $w$, then the pointwise intersection of the neighbourhoods $\mathcal{N}_i(w)$ for $i\in G$ will contain $\|\varphi\|^\mathfrak{M}$, but also (b) if $\neg\square_G\varphi$  is to be true at world $w$, then no pointwise intersection of sets in   $\mathcal{N}_i(w)$ for $i\in G$ will generate $\|\varphi\|^\mathfrak{M}$, i.e. we don't create too many intersection sets. To arrive at (b), we will need to make copies of each MCS $\Lambda$.



Let $\mathbb{G} = \{G \mid G\subseteq_f I\}$. Let $\mathbb{F}$ denote the set of all functions $f: \mathbb{G}\times\mathfrak{L}\rightarrow I$ such that, for all $G\in\mathbb{G}$ and all $\varphi\in\mathfrak{L}$, $f(G,\varphi) \in G$. The members of $\mathbb{F}$ are used as indices for the copies of the MCS in our canonical model:

\begin{definition}\label{def:canmodbase}
The \emph{canonical model for the base logic} is $\mathfrak{M}^c = \langle W^c, \langle \mathcal{N}^{c}_i\rangle_{i\in I}, V^c\rangle$, where

\begin{itemize}
\item[1.] $W^c = \{(\Lambda, f)\mid \Lambda \mbox{ is a MCS in } \mathfrak{L} \mbox{ and }  f\in \mathbb{F} \}$;
\item[2.] For all $\varphi\in \mathfrak{P}$, $V^c(\varphi) = \{(\Lambda,f)\in W^c\mid \varphi\in\Lambda\}$
\item[3.] for all $i\in I$, $\mathcal{N}^{c}_i(\Lambda,f) = \{X^{G,\varphi}_{i} \mid \square_G\varphi\in \Lambda, i\in G\subseteq_f I\}$ where,
\item[4.] for all $(G,\varphi)\in\mathbb{G}\times\mathfrak{L}$ and $i\in G$,

$$X^{G,\varphi}_{i} = \{(\Lambda,f) \in W^c \mid \varphi\in\Lambda \mbox{ or } f(G,\varphi)\neq i \}$$


\end{itemize}
\end{definition}

It is not hard to check that $\mathfrak{M}^c$ is well-defined; it suffices to show that $W^c$ is non-empty, which holds in view of the soundness of the base logic, and by a standard Lindenbaum construction.

The real difficulty consists in proving the truth lemma (Lemma \ref{lem:truth} below). To get there, we first prove two auxiliary lemmata:

\begin{lemma}\label{lem:aux}
Let $\mathcal{Y}$ be a set of sets $X^{G,\psi}_{i}$ with $i\in G$ and $(G,\psi)\in \mathbb{G}\times \mathfrak{L}$, such that for no $(G,\varphi)$, $\{X^{G,\varphi}_{i} \mid i\in G\}\subseteq \mathcal{Y}$. Then there is an $f_0\in\mathbb{F}$ such that
\begin{align} \{(\Lambda,f_0)\in W^c\}\subseteq \bigcap \mathcal{Y}
\end{align}
\end{lemma}
\begin{proof} Suppose the antecedent holds. Let $f\in\mathbb{F}$ be such that, for every $X^{G,\psi}_{i} \in\mathcal{Y}$, $f(G,\psi) = i^{G,\psi}$ for some $i^{G,\psi}\in G$ such that $X^{G,\psi}_{i^{G,\psi}} \not\in\mathcal{Y}$. In view of the supposition, there is at least one such $f$. Note that, for all $X^{G,\psi}_{i} \in\mathcal{Y}$, $f(G,\psi) \neq i$. By Definition \ref{def:canmodbase}, for all $X^{G,\psi}_{i} \in\mathcal{Y}$ and all MCS $\Lambda$, $(\Lambda,f)\in X^{G,\psi}_{i}$. Consequently, for all MCS $\Lambda$, $(\Lambda,f)\in \bigcap \mathcal{Y}$.
\end{proof}

\begin{lemma}\label{lem:aux2}
If $\mathcal{Y} = \{X^{G,\varphi}_{i} \mid i\in G\}$, then $\bigcap \mathcal{Y} = \{(\Lambda,f)\in W^c\mid \varphi\in \Lambda\}$.
\end{lemma}
\begin{proof}
By Definition \ref{def:canmodbase}.4,

\begin{align} \bigcap_{i\in G} X^{G,\varphi}_{i} = \bigcap_{i\in G} \{(\Lambda,f) \in W^c \mid \varphi\in\Lambda \mbox{ or } f(G,\varphi)\neq i \} \end{align}

In view of the definition of $\mathbb{F}$, we know that for every $i\in G$, there is some $f'\in\mathbb{F}$ such that $f'(G,\varphi) = i$. Hence,

\begin{align}\bigcap_{i\in G} \{(\Lambda,f) \in W^c \mid \varphi\in\Lambda \mbox{ or } f(G,\varphi)\neq i \} = \{(\Lambda,f)\in W^c\mid \varphi\in\Lambda\}\end{align}
\end{proof}

\begin{lemma}[Truth Lemma]\label{lem:truth} For all $(\Lambda,f)\in W^c$ and all $\varphi\in\mathfrak{L}$: $\mathfrak{M}^c,(\Lambda,f)\models \varphi$ iff $\varphi\in\Lambda$.
\end{lemma}
\begin{proof} By an induction on the complexity of $\varphi$. The base case and the induction step for the classical connectives are safely left to the reader. So it remains to prove that
\begin{itemize}
\item[] $\mathfrak{M}^c,(\Lambda,f)\models \square_G\varphi$ iff $\square_G\varphi\in\Lambda$ \hfill{(TL$\square$)}
\end{itemize}
\noindent \emph{Right to left direction of (TL$\square$).} Suppose that $\square_G\varphi\in\Lambda$. By Lemma \ref{lem:aux2},

\begin{align} \bigcap_{i\in G} X^{G,\varphi}_{i} = \{(\Lambda',f')\in W^c\mid \varphi\in\Lambda'\}\end{align}


So by the induction hypothesis (IH), we obtain:

\begin{align}\bigcap_{i\in G} X^{G,\varphi}_{i} = \|\varphi\|^{\mathfrak{M}^c}\end{align}

Moreover, by Definition \ref{def:canmodbase}.3, for every $i\in G$, $X^{G,\varphi}_{i} \in\mathcal{N}^c_i(\Lambda,f)$. By Definition \ref{def:model}, $\bigcap_{i\in G} X^{G,\varphi}_{i} \in \mathcal{N}_G(\Lambda,f)$. By Definition \ref{def:semclauses}, $\mathfrak{M}^c,(\Lambda,f)\models \square_G\varphi$.

\medskip

\noindent \emph{Left to right direction of (TL$\square$).} Suppose that $\mathfrak{M}^c,(\Lambda,f)\models \square_G\varphi$. So there is an $\mathcal{X} = \{ X_{i} \mid i\in G\}$ such that each $X_{i}\in \mathcal{N}^{c}_i(\Lambda,f)$ and
\begin{align} \bigcap \mathcal{X} = \|\varphi\|^{\mathfrak{M}^c}\label{eq:bigcap}\end{align}

By Definition \ref{def:canmodbase}.4, for all $i\in G$, $X_{i} = X^{H,\psi}_{j}$ for some $H,\psi$ and $j\in H$ such that $\square_H\psi\in \Lambda$. Let $\mathcal{A} = \{(H,\psi) \in\mathbb{G}\times\mathfrak{L} \mid X^{H,\psi}_{j}  \in\mathcal{X}\}$. Note that, since $G$ is finite, also $\mathcal{X}$ and $\mathcal{A}$ are finite.

We distinguish two cases:

\noindent Case 1: $\varphi$ is a tautology of the base logic. By the IH, $\|\varphi\|^{\mathfrak{M}^c} = W$. Hence each $X_{i} = W$. In view of Definition \ref{def:canmodbase}.4, for all $(H,\psi)\in\mathcal{A}$, $\psi$ is also a tautology and hence, by \axref{RE}, $\square_H\top \in \Lambda$. By \INTz, for all $j\in H$, $\square_j\top\in \Lambda$. It follows that $\square_i \top\in \Lambda$ for all $i\in G$. Since $G$ is finite, we can derive $\square_G\top$ using \INTe\ finitely many times. By \axref{RE}, $\square_G\varphi\in\Lambda$.

\noindent Case 2: $\varphi$ is not a tautology of the base logic. We first prove that, for some $K\subseteq G$, $\square_K\varphi\in \Lambda$. Let $\mathcal{B} = \{(H,\psi)\in\mathcal{A} \mid \mbox{ for all } j\in H, X^{N,\psi}_{j}\in \mathcal{X}\mbox{ and }\not\vdash\psi\}$. Note that\footnote{\label{reffootnote}To see why (\ref{eq:sqsubset}) and (\ref{eq:notcap}) hold, note that for each $(H,\psi)\in \mathcal{B}$ and each $j\in H$, there is a witness $X^{H,\psi}_{j}\in \mathcal{X}$ in view of the definition of $\mathcal{B}$. Note moreover that, by Definition \ref{def:canmodbase}.4, $X^{H,\psi}_{j}\neq X^{H',\psi'}_{j'}$ whenever $(H,\psi)\neq (H',\psi')$ or $j\neq j'$. So for every tuple $\epsilon = \langle H,\psi,j\rangle$ with $(H,\psi)\in\mathcal{B}$ and $j\in H$, there is a distinct $i_\epsilon \in G$ such that $X^{H,\psi}_{j} = X_{i_\epsilon}$.} 

\begin{align} \bigcup_{(H,\psi) \in \mathcal{B}} H \subseteq G \label{eq:sqsubset} \end{align}

\begin{align} \mbox{For all } (H,\psi), (H',\psi')\in\mathcal{B}: (H,\psi) = (H',\psi') \mbox{ or } H\cap H' = \emptyset \label{eq:notcap} \end{align}

We can now rewrite the intersection of the members of $\mathcal{X}$ as follows:

\begin{align}
\bigcap\mathcal{X} = \bigcap_{(H,\psi)\in\mathcal{A}\setminus \mathcal{B}, X^{H,\psi}_{j}\in \mathcal{X}} X^{H,\psi}_{j} \cap \bigcap_{(H,\psi)\in \mathcal{B}, j\in H} X^{H,\psi}_{j}
\end{align}

By Lemma \ref{lem:aux}, there is an $f'\in\mathbb{F}$ such that

\begin{align}
\{(\Lambda',f')\in W^c\}\subseteq \bigcap_{(H,\psi)\in\mathcal{A}\setminus \mathcal{B}, X^{H,\psi}_{j}\in \mathcal{X}} X^{H,\psi}_{j}
\end{align}

In view of Definition \ref{def:canmodbase}, Lemma \ref{lem:aux2}, and the IH,
\begin{align}
\bigcap_{(H,\psi)\in \mathcal{B}} \{(\Lambda',f') \in W^c\mid \psi \in\Lambda'\} = \{(\Lambda',f') \in W^c\mid \varphi\in\Lambda'\} \label{eq:Lambdakvarphi}
\end{align}
Hence, every MCS that contains every member of $\{\psi \mid (H,\psi)\in \mathcal{B}\}$ also contains $\varphi$, and vice versa. Since $\mathcal{B}$ is finite, this amounts to: 

\begin{align}\vdash\bigwedge_{(H,\psi)\in \mathcal{B}}\psi \leftrightarrow \varphi.\label{RefFML}\end{align}

By Definition \ref{def:canmodbase}.3 and the fact that $X_j^{H,\psi}\neq X_{j'}^{H',\psi}$ when $\psi$ is not a tautology and $j\neq j'$, $H\neq H'$ or $\psi\neq\psi'$ (cf.\ footnote \ref{reffootnote}) we have $\square_{H}\psi\in\Lambda$ for all $(H,\psi)\in \mathcal{B}$. Let $K  = \bigcup_{(H,\psi)\in \mathcal{B}} H$. Applying \INTe\, a suitable number of times, we can derive that $\square_{K} \bigwedge_{(H,\psi)\in \mathcal{B}}\psi \in \Lambda$. By \axref{RE} and (\ref{RefFML}),

\begin{align} \square_{K}\varphi\in \Lambda \label{eq:varphiinLambda}\end{align}

Let now $i\in G$. In view of the construction, there is an $X^{H_i,\psi_i}_{i} \in\mathcal{X}$ such that

\begin{align}
\square_{H_i} \psi_i \in \Lambda \label{eq:psiiinLambda}
\end{align}

Since $\bigcap\mathcal{X} = \|\varphi\|^{\mathfrak{M}^c}$, it follows that $X^{H_i,\psi_i}_{i} \supseteq \|\varphi\|^{\mathfrak{M}^c}$. Let $f_i\in \mathbb{F}$ be such that $f_i(H_i)=i$. Hence, $X^{H_i,\psi_i}_{i} \cap \{(\Lambda,f_i)\in W^c\} = \{(\Lambda,f_i)\in W^c\mid \psi_i\in \Lambda\}$. This implies that $\{(\Lambda,f_i)\in W^c\mid \psi_i\in \Lambda\} \supseteq \{(\Lambda,f_i) \mid \varphi\in\Lambda\}$, and hence
\begin{align}
\vdash \varphi \rightarrow \psi_i \label{eq:varphiveepsiinLambda}
\end{align}
This implies that,
\begin{align}
\vdash \psi_i \leftrightarrow (\varphi\vee \psi_i) \label{eq:varphiveepsiinLambda2}
\end{align}
By \axref{RE}, and since $i\in G$ was arbitrary, we have shown that for all $i\in G$, $\square_{H_i}(\varphi\vee\psi_i) \in \Lambda$. Let $G = \{i_1,\ldots,i_n\}$. Now we apply \INTs\ a $n$ times to derive $\square_{K\cup H_{i_1}} \varphi$, $\square_{K\cup H_{i_1}\cup H_{i_2}} \varphi$, etc., untill we finally arrive at $\square_{K\cup \bigcup_{i\in G} H_i} \varphi$. Note that $K\subseteq G\subseteq K\cup \bigcup_{i\in G}H_i$. From this and (\ref{eq:varphiinLambda}), we can derive that $\square_G \varphi\in\Lambda$ by \INTv.
\end{proof}

\section{Some Extensions}\label{sec:comp}

We now turn to a number of variants, obtained by imposing certain frame conditions on the neighbourhood functions $\mathcal{N}_i$. As will turn out, quite a number of additional frame conditions on the $\mathcal{N}_i$ do not impact the axiomatization for intersection neighbourhoods at all. Most results provided here will turn out to be relatively straightforward, building on our canonical model construction and completeness proof for the base logic.  {The proofs of all theorems in this section are slight adaptions of the argument for Theorem \ref{basethm}. We offer some details on the proofs   in the appendix.}



\subsection{Some Extensions on the Cheap}\label{sec:extcheap}

We first discuss some axioms, resp.\ frame conditions that require no changes in the construction of the canonical model, cf.\ Table \ref{table:cheapcond}. \axref{NEC} and \axref{P} are familiar from the study of Kripke-semantics. Adding \axref{CONEC} to any normal modal logic will result in a trivial system; adding \axref{COP} to any normal modal logic will result in a logic where the modal operator becomes useless (since $\square\varphi$ will be a theorem for all $\varphi$). However, in the context of non-normal modal logics, both axioms can sometimes make sense. The axiom \axref{CONEC} is not often mentioned; one of its concrete applications is in (non-normal) logics of agency \cite{Elgesem:1997}. The underlying idea is that an agent $i$ cannot (deliberately) bring about a tautology like ``the dishes are washed or they are not washed''. The axiom \axref{COP} has been used to characterize the notion of ``deontic sufficiency'' \cite{F:DNS}, often referred to as ``strong permission''. Here, $\square\varphi$ means that every $\varphi$-world is a permissible world; the axiom then follows trivially from the fact that no world verifies $\bot$.

As far as these conditions are concerned, our results are modular, in the sense that the frame conditions can be axiomatized independently; and we can moreover restrict each of them to certain groups $G$. This means that we can e.g.\ model cases where only one of the operators $\square_i$ satisfies necessitation, whereas the others do not. 


\verberg{
\begin{table}
\begin{tabular}{|ccc|}

\hline



(NEC$_G$) & $W\in \mathcal{N}_G(w)$ & $\vdash \square_G \top$  \\

(CONEC$_G$) & $W\not\in \mathcal{N}_G(w)$ & $\vdash \neg\square_G \top$ \\

\axref{P$_G$} & $\emptyset\not\in \mathcal{N}_G(w)$ & $\vdash \neg \square_G \bot$ \\

(COP$_G$) & $\emptyset\in \mathcal{N}_G(w)$ & $\vdash \square_G \bot$ \\

\hline

\end{tabular}
\caption{Some extensions on the cheap. Here, we always quantify universally over $w$.}\label{table:cheapcond}
\end{table}
}

\begin{table}


\begin{center}
\begin{minipage}{.6\textwidth}
\begin{align*}
&W\in \mathcal{N}_i(w) && \vdash \square_i \top\axlabel{NEC}  \\[.3em]
 &W\not\in \mathcal{N}_i(w) && \vdash \neg\square_i \top\axlabel{CONEC} \\[.3em]
 &\emptyset\not\in \mathcal{N}_i(w) && \vdash \neg \square_i \bot\axlabel{P}  \\[.3em]
 &\emptyset\in \mathcal{N}_i(w) && \vdash \square_i \bot\axlabel{COP}
&\end{align*}
\end{minipage}
\end{center}
\caption{Some extensions on the cheap. Here, $i$ ranges over all agents in $I$ and we always quantify universally over $w$.}\label{table:cheapcond}
\end{table}

\begin{theorem}\label{th:SCcheap} The logic of any selection of frame conditions from Table \ref{table:cheapcond} is axiomatized by adding the corresponding axioms from that table to the base logic.\end{theorem}

We should highlight that in the cases of \axref{NEC}, \axref{CONEC} and \axref{COP}, the corresponding frame condition also holds for the $\mathcal{N}_G(w)$. For instance, as soon as $W\in \mathcal{N}_i(w)$ for all $i\in I$, we can infer that $W\in\mathcal{N}_G(w)$ for all $G\subseteq_f I$. At the syntactic level, this is mirrored by the following property:

\begin{lemma}\label{inheritancelem}
For any extension $\vdash$ of the base logic: if for all $i\in I$, $\vdash \square_i \top$ (resp.  $\vdash \neg\square_i \top$ or $\vdash \square_i \bot$ ), then for all $G\subseteq_f I$, $\vdash \square_G \top$ (resp.  $\vdash \neg\square_G \top$ or $\vdash \square_G \bot$ ).
\end{lemma}

In other words, the three mentioned frame conditions and the corresponding axioms readily transfer from single indices to groups. Consequently, imposing these frame conditions on groups rather than individual indices will not make any difference to the logic.

This is not true for \axref{P}. It is easy to construct a model with $I=\{1,2\}$ where $\emptyset\not\in\mathcal{N}_i(w)$ for $i\in \{1,2\}$ and all $w$, but $\emptyset\in\mathcal{N}_{\{1,2\}}(w)$ for some (or even all) $w$.

\begin{theorem}\label{th:SCcheapfor all } The logic of frame condition $\emptyset\not\in \mathcal{N}_G(w)$ in conjunction with any selection of frame conditions from Table \ref{table:cheapcond} is axiomatized by adding to the base logic the corresponding axioms from that table and all instances of the following axiom schema:
\begin{align*}
\neg\square_G\bot\axlabel{P$_G$}
\end{align*}
\end{theorem}

Some combinations of the axioms from Table \ref{table:cheapcond} obviously result in a trivial logic if we use the same $G$ everywhere. Note also that adding \axref{NEC} to the base logic allows us to derive the following theorem, using \INTe:

\begin{align*}
\square_G \varphi\rightarrow \square_{G\cup H}\varphi \axlabel{SA}
\end{align*}


\axref{SA} stands for \emph{superadditivity}, which is the common name used for this type of axiom in logics of (group) agency, (distributed) belief, and (distributed) knowledge. In the presence of \axref{SA}, the axioms \INTz, \INTv, and \INTs\, become derivable. So we obtain a very simple alternative characterization of the logic of all models where, for all $i\in I$, $W\in \mathcal{N}_i(w)$: all we need is \INTe\, and \axref{NEC}.



\subsection{The T-schema}

In the remainder of this section, we will point out a few completeness results that are less modular, in the sense that they concern frame conditions that are imposed on all the neighbourhoods $\mathcal{N}_i(w)$ for all $i\in I$ at once, rather than for a selection of them. We start with the T-schema: $\square\varphi\rightarrow \varphi$. Let us call a neighbourhood function \emph{reflexive} iff, for every $w\in W$ and for every $X\in\mathcal{N}(w)$, $w\in X$.

\begin{theorem}\label{reflthm}
The logic of the class of models $\mathfrak{M} = \langle W,\langle\mathcal{N}_i\rangle_{i\in I},V\rangle$ where each $\mathcal{N}_i$ is reflexive is axiomatized by adding to the base logic all instances of the following axiom schema:

\begin{align*}
\square_G\varphi \rightarrow \varphi \axlabel{T$_G$}
\end{align*}
\end{theorem}

\noindent Importantly, one cannot get a complete axiomatization of reflexivity by just adding the axioms (T$_i$), i.e.\, $\square_i\varphi\rightarrow \varphi$ to the base logic. To see this, note that all axioms $\square_G\phi\rightarrow\phi$ for  $G\subseteq_f I$ are sound with respect to reflexive frames. The following example of a non-reflexive frame shows that these axioms do not logically follow from $\square_i\varphi\rightarrow\varphi$. We consider a simple case with   $I = \{1,2\}$. Take a model $\mathfrak{M}$ with two worlds, $w$ and $v$, where all propositional formulas are true at both worlds. Suppose now that $\mathcal{N}_1(w) = \mathcal{N}_1(v)= \{\{w\}\}$ and $\mathcal{N}_2(w) = \mathcal{N}_2(v) = \{\{v\}\}$. Since neither $\{w\}$ nor $\{v\}$ correspond to the truth set of any formula $\varphi$ in this model, $\square_1 \varphi$ and $\square_2\varphi$ will be false for every $\varphi$, and hence (T$_1$) and (T$_2$) will be trivially valid in this model. However, this model does not validate (T$_{\{1,2\}}$), since $\square_{\{1,2\}}\bot$ is true at $w$ and at $v$. So the model satisfies all formulas of the form $\square_i\phi\rightarrow\phi$ together with \INTe\,-\INTz, but not $\square_G\phi\rightarrow\phi$:

\subsection{Binary Consistency}

In any normal modal logic, \axref{P} is equivalent to the following axiom:

\begin{align*}
\square\varphi\rightarrow \neg\square\neg\varphi\axlabel{D}
\end{align*}
However, in neighbourhood models, the two axioms are non-equivalent. Whereas \axref{P} expresses that $W\not\in \mathcal{N}(w)$, \axref{D} expresses that if $X\in \mathcal{N}(w)$, then $W\setminus X\not\in \mathcal{N}(w)$. It can easily be verified that, by adding indexed variants of the \axref{D}-axiom, we get a complete logic for all frames that satisfy the following frame condition:
\begin{itemize}
\item[] \emph{Binary consistency}: for all $i\in I$: if $X\in\mathcal{N}_i(w)$, then $W\setminus X\not\in \mathcal{N}_i(w)$
\end{itemize}

\begin{theorem}\label{binconsthm}
The logic of the class of models $\mathfrak{M} = \langle W,\langle\mathcal{N}_i\rangle_{i\in I},V\rangle$ that satisfy binary consistency is axiomatized by adding to the base logic all instances of the following axiom schema, for all $i\in I$:

\begin{align*}
\square_i\varphi \rightarrow \neg\square_i\neg\varphi\axlabel{D$_i$}
\end{align*}
\end{theorem}

\subsection{Closure under Supersets}\label{subsec:closuresubsets}

We call a model $\mathfrak{M} = \langle W,\langle \mathcal{N}_i\rangle_{i\in I}, V\rangle$ \emph{monotone} iff, for all $i\in I$ and all $w\in W$, $\mathcal{N}_i(w)$ is closed under supersets. This means that for all $X \in\mathcal{N}_i(w)$, for all $Y\subseteq W$ with $X\subseteq Y$, also $Y\in\mathcal{N}_i(w)$.

\begin{theorem}\label{closesuper}
The logic of the class of all monotone models is axiomatized by adding to the base logic all instances of the following axiom schema:
\begin{align*}
\square_G\varphi \rightarrow \square_G(\varphi \vee\psi) \axlabel{RM$_G$}
\end{align*}
\end{theorem}

\noindent {Here we  slightly deviate from our standard canonical model construction (Definition \ref{def:canmodbase}). To ensure that the canonical model falls in the class of monotone models, we need to close all neighbourhoods under supersets (cf.\ Definition \ref{def:canmodsuper} in the Appendix).}

Note that in the presence of \axref{RM$_G$}, \axref{RE} becomes a derived rule. Also, it can easily be observed that if we add any (consistent) combination of the axioms (T), \axref{P}, \axref{P$_G$}, \axref{NEC} to the base logic + \axref{RM$_G$}, then we can prove that the associated canonical model $\mathfrak{M}^c_\uparrow$ will be monotone and satisfy the associated frame condition.


\subsection{Closure under Finite and Infinite Intersections}\label{subsec:intersect}

In regular neighbourhood modal logic with one modality $\square$, closure of the neighbourhood function under finite intersections yields the axiom of aggregation: $(\square \varphi\wedge\square\psi) \rightarrow \square(\varphi\wedge\psi)$. In fact, the logic obtained by adding \axref{RE} and (C) to classical logic is complete for both, the class of frames where the neighbourhood function is closed under finite intersections, and the class of frames where the neighbourhood function is closed under infinite intersections. We now generalize this fact to neighbourhood models with pointwise intersection:

\begin{theorem}\label{intersectthm}
The logic of the class of all models where each $\mathcal{N}_i(w)$ is closed under \old{infinite}{arbitrary} intersections is axiomatized by replacing, in the base logic, the axiom \INTe\, with its unrestricted counterpart:
\begin{align*}
(\square_{G}\varphi \wedge\square_H\psi) \rightarrow \square_{G\cup H}(\varphi \wedge\psi) \axlabel{C$_G$}
\end{align*}
\end{theorem}

\noindent{For the proof, again, we have to deviate slightly from our canonical model construction for the base logic, by closing neighbourhoods under arbitrary intersection.} Note that \axref{C$_G$} is also sound for the class of models where the neighbourhood sets are closed under finite intersection. So we immediately obtain:

\begin{corollary}
The logic of the class of all models where each $\mathcal{N}_i(w)$ is closed under finite intersections is axiomatized by adding to the base logic all instances of \axref{C$_G$}.
\end{corollary}

\section{Summary and Outlook}\label{sec:concl}

In this paper, we axiomatized the base logic of neighbourhood models with pointwise intersection and various extensions obtained by imposing standard frame conditions on the neighbourhoods for the individual indexes. For the canonical model construction in our completeness proof we made use of a copying technique that is -- as far as we know -- new. In forthcoming work, we generalize these results, including the operation of pointwise intersection of a neighbourhood set with \emph{itself} and establishing the finite model property for the resulting classes of logics.

Some obvious open questions concern the other (standard) frame conditions that correspond to well-known axioms such as the (4)-axiom, the (5)-axiom, and other ``usual suspects'' in modal logic. Also, one may consider the possibility of adding a universal modality to the logics, which in turn allows one to express conditions like monotonicity by means of axioms schemata. Finally, one may consider multi-modal logics where only some of the individual operators satisfy certain principles (e.g.\ one non-normal operator for ability, and another normal operator for belief or knowledge), and check to what extent our current techniques can be applied to those.

Our definition of the canonical model, we conjecture, can be easily generalized to axiomatize other operations on neighbourhood functions. One may e.g.\ define \emph{pointwise union} in a wholly analogous fashion, replacing every occurrence of $\cap$ in Definition \ref{def:model} with $\cup$. Drawing inspiration from Dynamic Logic, one may also define various operations of sequential composition of neighbourhoods. In sum, we believe that the perspective we have tried to sketch here allows for a plethora of fascinating new logical investigations and philosophical applications.

\bibliographystyle{aiml18}


\begin{thebibliography}{10}
\expandafter\ifx\csname url\endcsname\relax
  \def\url#1{\texttt{#1}}\fi
\expandafter\ifx\csname urlprefix\endcsname\relax\def\urlprefix{URL }\fi
\newcommand{\enquote}[1]{``#1''}

\bibitem{AagotnesDistBel}
{\AA}gotnes, T. and Y.~N. W\'ang, \emph{Resolving distributed knowledge},
  Artificial Intelligence \textbf{252} (2017), pp.~1 -- 21.

\bibitem{Belnapetal:ff}
Belnap, N., P.~M., X.~M. and B.~P., \enquote{Facing the Future: Agents and
  Choice in Our Indeterminist World,} Oxford University Press, 2001.

\bibitem{Broersenetal:nsCL}
Broersen, J., A.~Herzig and N.~Troquard, \emph{A normal simulation of coalition
  logic and an epistemic extension}, in: D.~Samet, editor, \emph{Proceedings of
  the 11th Conference on Theoretical Aspects of Rationality and Knowledge},
  TARK '07 (2007), pp. 92--101.

\bibitem{Brown:logicab}
Brown, M., \emph{On the logic of ability}, Journal of Philosophical Logic
  \textbf{17} (1988), pp.~1--26.

\bibitem{C&J2002}
Carmo, J. M. C. L.~M. and A.~J.~I. Jones, \emph{Deontic logic and
  contrary-to-duties}, in D. Gabbay and F. Guenthner, editors, Handbook of Philosophical Logic, Vol.\  \textbf{8}, Kluwer Academic Publishers, 2002, 2nd
  edition, pp. 147--264.

\bibitem{Chellas:ml}
Chellas, B., \enquote{Modal Logic: an Introduction,} Cambridge: Cambridge
  university press, 1980.

\bibitem{Elgesem:1997}
Elgesem, D., \emph{The modal logic of agency}, Nordic J. Philos. Logic
  \textbf{2} (1997), p.~146.

\bibitem{faginetal:RAK}
Fagin, R., J.~Y. Halpern, Y.~Moses and M.~Y. Vardi, \enquote{Reasoning About
  Knowledge,} MIT Press, Cambridge, Massachusetts, 2003.

\bibitem{OUR}
Faroldi, F. L.~G. and T.~Protopopescu, \emph{Hyperintensional logics of
  reasons} (2017), ms.

\bibitem{GargovPassy:booleanmodallogic}
Gargov, G. and S.~Passy, \emph{A note on boolean modal logic}, in: P.~P.
  Petkov, editor, \emph{Mathematical Logic}, Springer US, 1990 pp. 299--309.

\bibitem{Gargov-booleanspec}
Gargov, G., S.~Passy and T.~Tinchev, \emph{Modal environment for boolean
  speculations}, in: \emph{Mathematical logic and its applications}, Plenum
  Press, 1986 .

\bibitem{Gerbrandky:DK}
Gerbrandy, J., \emph{Distributed knowledge}, in: J.~Hulstijn and A.~Nijholt,
  editors, \emph{Twendial 98: Formal Semantics and Pragmatics of Dialogue,
  TWLT 13}, Universiteit Twente, Enschede, 1998, pp. 111--124.

\bibitem{LG:ldd}
Goble, L., \emph{A logic for deontic dilemmas}, Journal of Applied Logic
  \textbf{3} (2005), pp.~461--483.

\bibitem{Goblehandbook}
Goble, L., \emph{Prima facie norms, normative conflicts, and dilemmas}, in D. Gabbay,
L. van der Torre, J. Horty, and X. Parent, editors, Handbook of Deontic Logic and
Normative Systems, Vol.\ 1, College Publications, 2013 pp. 241--351.

\bibitem{Governatori:2005}
Governatori, G. and A.~Rotolo, \emph{On the axiomatisation of elgesem's logic
  of agency and ability}, Journal of Philosophical Logic \textbf{34} (2005),
  pp.~403--431.

\bibitem{Hareletal:2000}
Harel, D., D.~Kozen and J.~Tiuryn, \enquote{Dynamic Logic,} Cambridge, MA: MIT
  Press, 2000.

\bibitem{HerzigSchwartzentruber:prop}
Herzig, A. and F.~Schwarzentruber, \emph{Properties of logics of individual and
  group agency}, in: C.~Areces and R.~Gobldblatt, editors, \emph{Advances in
  Modal Logic} (2008).

\bibitem{Horty:adl}
Horty, J.~F., \enquote{Agency and Deontic Logic,} Oxford University Press, New
  York, 2001.

\bibitem{KRG}
Klein, D., O.~Roy and N.~Gratzl, \emph{Knowledge, belief, normality, and
  introspection}, Synthese  (2017), pp.~1--30.

\bibitem{DL:c}
Lewis, D., \enquote{Counterfactuals,} Harvard University Press, Cambridge,
  Mass., 1973.

\bibitem{Montague:1970}
Montague, R., \emph{Universal grammar}, Theoria \textbf{36} (1970),
  pp.~373--398.

\bibitem{NairHorty:fc}
Nair, S. and J.~F. Horty, \enquote{The Oxford Handbook of Reasons and
  Normativity,} USA: Oxford University Press, forthcoming .

\bibitem{Pacuit:2017}
Pacuit, E., \enquote{Neighbourhood Semantics for Modal Logic,} Springer, 2017.

\bibitem{Pauly:2002}
Pauly, M., \emph{A modal logic for coalitional power in games}, Journal of
  Logic and Computation \textbf{1} (2002), pp.~149--166.

\bibitem{Scott1970}
Scott, D., \enquote{Advice on Modal Logic,} Springer Netherlands, Dordrecht,
  1970 pp. 143--173.

\bibitem{Segerberg:1971}
Segerberg, K., \emph{An essay in classical modal logic} (1971).

\bibitem{StalnakerKB}
Stalnaker, R., \emph{On logics of knowledge and belief}, Philosophical Studies
  \textbf{128} (2006), pp.~169--199.

\bibitem{vanBenthem&Pacuit:2011}
van Benthem, J. and E.~Pacuit, \emph{Dynamic logics of evidence-based beliefs},
  Studia Logica \textbf{99} (2011), pp.~61--92.

\bibitem{F:DNS}
Van De~Putte, F., \emph{That will do: Logics of deontic necessity and
  sufficiency}, Erkenntnis \textbf{82} (2017), pp.~473--511.

\end{thebibliography}

\newpage
\appendix
\setcounter{section}{1}
For convenience, we restate theorems and lemmas before proving them.

\begin{rlemma}{Indeplem}
Axioms \INTe-\INTs\ are logically independent from each other.
\end{rlemma}
\begin{proof} {We sketch the argument that \INTe-\INTs\ are mutually independent. In view of the soundness of these axioms w.r.t.\ models with pointwise intersection, we can only falsify those axioms in models of a more general type, i.e.\ where each of the neighbourhood functions $\mathcal{N}_G$ are primitive. We stick to the semantic clauses from Definition \ref{def:semclauses}. All our examples work with a set of agents $I=\{1,2,3\}$ and a set of worlds  $W=\{w_p,w_q,w_r\}$, where the atoms $p,q,r$ are true at $w_p,w_q$ and $w_r$ respectively. The models we construct only differ in their neighbourhood functions. In the following, whenever a neighbourhood $\N_G(w)$ for $G\subseteq \{1,2,3\}$ remains  unspecified, we assume that $\N_G(w)=\{\emptyset\}$. Moreover, all neighbourhood functions are assumed constant, i.e. $\N_G(w)=\N_G(w')$ for all $w,w'\in W$. We will   write $\N$ instead of $\N(w)$.}

To see that \INTe\ is independent of \INTz-\INTs, we define model $\M_1$ as follows: Let $\N_{\{1\}}=\{\{w_p,w_r\},\emptyset\}$ and  $\N_{\{2\}}=\{\{w_q,w_r\},\emptyset\}$. It is easy to check that  \INTz-\INTs\ are valid on this model. {First, the antecedent of \INTz\, is always false. Second, for \INTv\, and \INTs, the antecedent can only be true if $\|\varphi\|^{\M_1} = \emptyset$; under this condition, the consequent is easily verified.} However, we have that $\M_1,w_p\models \square_{\{1\}}(p\vee r)\wedge \square_{\{2\}}(q\vee r)$ but  $\M_1,w_p\not\models \square_{\{1,2\}}((p\vee r)\wedge(q\vee r))$, contradicting \INTe.

Next, to show that \INTz\ is independent of \INTe,\INTv\ and \INTs, define the model $\M_2$ by taking neighbourhoods to be $\N_{\{i\}}=\wp(W)\setminus\{1,2,3\}$ for all singletons $\{i\}$ and $\N_{G}=\wp(W)$ for all $G\subseteq I$ of cardinality at least $2$. {Note that for all $\varphi$ and all groups $G$ with cardinality at least $2$, $\square_G \varphi$ is true at all worlds in $\M_2$.} From this one can easily infer that \INTe,\INTv\ and \INTs\ are valid in $\M_2$. However, we have $\M_2,w_p\models\square_{\{1,2\}}\top\wedge\neg\square_{\{1\}}\top$ contradicting \INTz.

To see that \INTv\ is independent from \INTe,\INTz\ and \INTs\ consider  model $\M_3$ with neighbourhood $N_{\{1\}} =N_{\{1,2,3\}} =\{\{w_p\},\emptyset\} $. Again it's easy to see that this model satisfies \INTe,\INTz\ and \INTs, but not \INTv\ as $\M_3, w_p\models\square_{\{1\}}p\wedge \square_{\{1,2,3\}}p$ but $\M_3, w_p\not\models\square_{\{1,2\}}p$.

To see that \INTs\ is independent of \INTe-\INTv\ consider model $\M_4$ with  neighbourhoods   $N_{\{1,3\}}=\{\{w_p\},\emptyset\} $ and $N_{\{1,2\}}=N_{\{1,2,3\}}=\{\{w_p,w_q\},\emptyset\}$.  It is  easy to see that this model satisfies \INTe-\INTv, but not \INTs, as $\M_4,w_p\models\square_{\{1,3\}}p\wedge\square_{\{1,2\}}p\vee q$, but $\M_4,w_p\not\models\square_{\{1,2,3\}}p$.
\end{proof}

\begin{rlemma}{inheritancelem}
For any extension $\vdash$ of the base logic: if for all $i\in I$, $\vdash \square_i \top$ (resp.  $\vdash \neg\square_i \top$ or $\vdash \square_i \bot$ ), then for all $G\subseteq_f I$, $\vdash \square_G \top$ (resp.  $\vdash \neg\square_G \top$ or $\vdash \square_G \bot$ ).
\end{rlemma}

\begin{proof}
Assume $\vdash \square_i \top$ for all $i\in I$ and let $G\subseteq_f I$. Then an iterated application of \INTe\ yields $\vdash \square_G \top$. Likewise, if $\vdash \square_i \bot$,  an iterated application of \INTe \ yields $\vdash \square_G \bot$. Finally for  $\vdash \neg\square_i \top$  note that by \INTz, we have that $\vdash\square_G\top\rightarrow\square_j\top$ for all $G\subseteq_f I$ and $j\in G$. Hence. $\vdash\bigwedge_{j\in G}\neg\square_j \top\rightarrow\neg\square_G\top$.
\end{proof}

\begin{rtheorem}{th:SCcheap} The logic of any selection of frame conditions from Table \ref{table:cheapcond} is axiomatized by adding the corresponding axioms from that table to the base logic.\end{rtheorem}
\begin{proof}This is  a straightforward adaption of the original proof. The only additional thing to show is that the $X_i^{G,\phi}$  do not violate any of the four frame conditions. For \axref{NEC} and \axref{COP} this is immediate. For \axref{P} it follows from the fact that $\emptyset\subset X_i^{G,\phi}$ whenever $\not\vdash\bot\leftrightarrow\phi$ or $|G|>2$, together with $\Box_i\bot\not\in\Lambda$ for any $\Lambda$. For \axref{CONEC} it follows from the fact that $X_i^{G,\phi}\subset W^c$ whenever $\not\vdash\phi\leftrightarrow\top$ together with Lemma \ref{inheritancelem}.\end{proof}

\begin{rtheorem}{th:SCcheapfor all } The logic of frame condition $\emptyset\not\in \mathcal{N}_G(w)$ in conjunction with any selection of frame conditions from Table \ref{table:cheapcond} is axiomatized by adding to the base logic the corresponding axioms from that table and all instances of \axref{P$_G$}.\end{rtheorem}

\begin{proof} Soundness is a matter of routine: one simply checks that the axiom is valid whenever the corresponding frame condition holds. 

We briefly sketch the completeness proof for \axref{P}; for each of the other three axioms the reasoning is completely analogous. First, we construct the canonical model according to Definition \ref{def:canmodbase}, with the only difference that our maximal consistent sets are constructed using the stronger logic that also contains the \axref{P}-axiom. We then prove the auxiliary lemmata and the truth lemma, just as before (see Lemmas \ref{lem:aux}, \ref{lem:aux2}, and \ref{lem:truth}). By the Truth Lemma, we obtain that $\emptyset\in\mathcal{N}_G(\Lambda,f)$ iff $\square_G\bot\in \Lambda$. However, for all MCS $\Lambda$, we also know that $\neg\square_G\bot\in\Lambda$. Hence, since every such $\Lambda$ is consistent, we can infer that $\emptyset\not\in\mathcal{N}_G(\Lambda,f)$.
\end{proof}

\begin{rtheorem}{reflthm}
The logic of the class of models $\mathfrak{M} = \langle W,\langle\mathcal{N}_i\rangle_{i\in I},V\rangle$ where each $\mathcal{N}_i$ is reflexive is axiomatized by adding to the base logic all instances of  the following axiom schema:
\begin{align*}
\square_G\varphi \rightarrow \varphi\axlabel[T$_G$]{TG-App}
\end{align*}
\end{rtheorem}
\begin{proof} Soundness is again a matter of routine. For completeness, we can again use the canonical model construction from Definition \ref{def:canmodbase}. The auxiliary lemmata and the truth lemma are proven as before; it suffices to show that the frame condition for \axref{TG-App} is satisfied. So suppose that $X^{G,\varphi}_{i}$ is a member of $\mathcal{N}_i(\Lambda,f)$. In view of the construction, (a) $X^{G,\varphi}_{i}$ is a superset of the set $\{(\Lambda',f')\in W^c \mid \varphi\in\Lambda\}$ and (b) $\square_G\varphi \in\Lambda$. By (b) and the axiom \axref{TG-App}, also $\varphi\in\Lambda$, and hence by (a), for all $f''\in \mathbb{F}$, $(\Lambda,f'') \in X^{G,\varphi}_{i}$. Consequently, $(\Lambda,f) \in X^{G,\varphi}_{i}$.
\end{proof}

\begin{rtheorem}{binconsthm}
The logic of the class of models $\mathfrak{M} = \langle W,\langle\mathcal{N}_i\rangle_{i\in I},V\rangle$ that satisfy binary consistency is axiomatized by adding to the base logic all instances of  the following axiom schema, for all $i\in I$:
\begin{align*}
\square_i\varphi \rightarrow \neg\square_i\neg\varphi \axlabel[D$_i$]{DI-Append}
\end{align*}
\end{rtheorem}
\begin{proof} We can again use the same canonical model construction. It suffices to show that in the presence of \axref{DI-Append}, this model will satisfy binary consistency. So suppose that $i\in I$, $(\Lambda,f)\in W^c$ and $X\subseteq W^c$ are such that $X, Y\in \mathcal{N}_i(\Lambda,f)$ where $Y = W\setminus X$. Case 1: $X$ is definable, i.e.\ there is some $\varphi$ such that $X = \|\varphi\|^{\mathfrak{M}^c}$. In that case, by the truth lemma, $\square_i \varphi \wedge\square_i\neg\varphi\in \Lambda$, contradicting the supposition that $\Lambda$ is consistent and closed under \axref{DI-Append}.

Case 2: $X$ and $Y$ are not definable.  Note that by the construction of $\mathfrak{M}^c$, $X = X^{G,\varphi}_{i}$ and $Y = Y^{H,\psi}_{i}$, with $i\in G\cap H$. Suppose first that $G=\{i\}$ or $H=\{i\}$. Then by Definition \ref{def:canmodbase} and the truth lemma, $X = \|\varphi\|^{\mathfrak{M}^c}$ or $Y = \|\psi\|^{\mathfrak{M}^c}$, contradicting the assumption that neither $X$ nor $Y$ are definable. So there are $j,k$ such that $j\in G\setminus \{i\}$ and $k\in H\setminus \{i\}$. Let now $f'\in\mathbb{F}$ be such that $f'(G,\varphi) = j$ and $f'(H,\varphi) = k$, and let $\Lambda$ be an arbitrary MCS. Note that $(\Lambda,f') \in X^{G,\varphi}_{i} \cap Y^{H,\psi}_{i}$ by the construction of $\mathfrak{M}^c$. Hence, $X \cap Y\neq\emptyset$, contradicting the supposition that $Y = W\setminus X$.
\end{proof}

\begin{rtheorem}{closesuper}
The logic of the class of all monotone models is axiomatized by adding to the base logic all instances of  the following axiom schema:
\begin{align*}
\square_G\varphi \rightarrow \square_G(\varphi \vee\psi) \axlabel[RM$_G$]{RMG-App}
\end{align*}
\end{rtheorem}
\begin{proof} Soundness is a matter of routine. For completeness, we need a slightly different construction. Let $\mathfrak{M}^c = \langle W^c,\langle\mathcal{N}^c_i\rangle_{i\in I}, V^c\rangle$ be defined as before -- see Definition \ref{def:canmodbase}. Now, define $\mathfrak{M}^c_{\uparrow}$ as follows:

\begin{definition}\label{def:canmodsuper}
$\mathfrak{M}^c_{\uparrow} = \langle W^c, \langle \mathcal{N}^{c\uparrow}_i\rangle_{i\in I}, V^c\rangle$, where for all $(\Lambda,f)\in W^c$, $\mathcal{N}^{c\uparrow}_i(\Lambda,f)$ is the closure of $\mathcal{N}^c_i(\Lambda,f)$ under supersets: $\mathcal{N}^{c\uparrow}_i(\Lambda,f) = \{Y\subseteq W^c \mid \mbox{ for an } X\in\mathcal{N}^c_i(\Lambda,f), X\subseteq Y\}$.
\end{definition}

Note that lemmas \ref{lem:aux} and \ref{lem:aux2} are preserved, since these only concern the sets $X^{G,\varphi}_{i}$ that are used in the construction of each $\mathcal{N}_i$. The truth lemma however needs to be proved anew. Again, the crucial point is to prove the induction step for $\square_G$:
\begin{align*}
\mathfrak{M}^c_\uparrow,(\Lambda,f)\models \square_G\varphi\text{ iff }\square_G\varphi\in\Lambda\axlabel[TL$\square \uparrow$]{TLSU}
\end{align*}

For the right-to-left direction of (\ref{TLSU}), we can simply repeat the proof of the right-to-left direction of (TL$\square$). For left-to-right, some small changes are required, which we spell out here.

Suppose that $\mathfrak{M}^c_{\uparrow},(\Lambda,f)\models \square_G\varphi$. By the semantic clause for $\square_G$, there is a $\mathcal{Z} = \{Z_{i}\mid i\in G\}$ such that each $Z_{i}\in \mathcal{N}^{c\uparrow}_i(\Lambda,f)$ and $\bigcap_{i\in G} Z_{i} = \|\varphi\|^{\mathfrak{M}^c_\uparrow}$. By the construction, for each $Z_{i}\in\mathcal{Z}$ there is an $X_{i} \in \mathcal{N}_i(\Lambda,f)$ such that $X_{i}\subseteq Z_{i}$. Hence,

\begin{align} \bigcap_{i\in G} X_{i}\subseteq \|\varphi\|^{\mathfrak{M}^c_\uparrow} \end{align}

We define $\mathcal{A}$ as before. 
Note that, in view of the preceding, each $\mathcal{N}^{c\uparrow}_i(\Lambda,f)$ with $i\in G$ is non-empty. This implies that for all $i\in G$, there is some $\psi_i$ and some $G_i$ that contains $i$, such that $\square_{G_i}\psi_i \in\Lambda$. By \INTz, $\square_i\psi_i\in \Lambda$ and hence by \axref{RMG-App}, also

\begin{align} \square_i\top\in\Lambda \mbox{ for all } i\in G \label{eq:top} \end{align}

Case 1: $\varphi$ is a tautology. By (\ref{eq:top}), using \INTe\, $\square_{G} \top\in\Lambda$. By \axref{RE}, $\square_G\varphi\in \Lambda$.

Case 2: $\varphi$ is not a tautology. Define $\mathcal{B}$ as before. We can now reason just as before, but instead of deriving an identity, we get at the following set inclusion: 

\verberg{
\begin{align}
\bigcap \mathcal{Z}\supseteq \bigcap \mathcal{X} = \bigcap_{(H,\psi)\in\mathcal{A}\setminus \mathcal{B}, X^{H,\psi}_{j}\in \mathcal{X}} X^{H,\psi}_{j} \cap \bigcap_{(H,\psi)\in \mathcal{B}, j\in H} X^{H,\psi}_{j}
\end{align}

By Lemma \ref{lem:aux}, there is an $f'\in\mathbb{F}$ such that

\begin{align}
\{(\Lambda',f')\in W^c\}\subseteq \bigcap_{(H,\psi)\in\mathcal{A}\setminus \mathcal{B}, X^{H,\psi}_{j}\in \mathcal{X}} X^{H,\psi}_{j}
\end{align}
In view of Definition \ref{def:canmodbase}, Lemma \ref{lem:aux2}, and the IH,
}

\begin{align}
\bigcap_{(H,\psi)\in \mathcal{B}} \{(\Lambda',f') \in W^c\mid \psi \in\Lambda'\} \subseteq \{(\Lambda',f') \in W^c\mid \varphi\in\Lambda'\}
\end{align}
Hence, every MCS that contains every member of $\{\psi \mid (H,\psi)\in \mathcal{B}\}$ also contains $\varphi$. Since $\mathcal{B}$ is finite, this gives us:

\begin{align}\vdash\bigwedge_{(H,\psi)\in \mathcal{B}}\psi \rightarrow \varphi.\label{RefFMLRM}\end{align}

By Definition \ref{def:canmodbase}.3, $\square_{H}\psi\in\Lambda$ for all $(H,\psi)\in \mathcal{B}$. Let $K  = \bigcup_{(H,\psi)\in \mathcal{B}} H$. Note that, by (\ref{eq:sqsubset}), $K \subseteq G$. Applying \INTe\, a suitable number of times, we can derive that $\square_{K} \bigwedge_{(H,\psi)\in \mathcal{B}}\psi \in \Lambda$. By \axref{RMG-App} and (\ref{RefFMLRM}),

\begin{align} \square_{K }\varphi\in \Lambda \label{eq:varphiinLambdaRM}\end{align}

From there, we can follow the exact same reasoning as that in the proof for the base logic, starting after equation (\ref{eq:varphiinLambda}).
\end{proof}

\begin{rtheorem}{intersectthm}
The logic of the class of all models where each $\mathcal{N}_i(w)$ is closed under \old{infinite}{arbitrary} intersections is axiomatized by replacing, in the base logic, the axiom \INTe\, with its unrestricted counterpart:
\begin{align*}
(\square_{G}\varphi \wedge\square_H\psi) \rightarrow \square_{G\cup H}(\varphi \wedge\psi) \axlabel[C$_G$]{CG-Appendix}
\end{align*}
\end{rtheorem}
\begin{proof} Soundness is again a matter of routine. For completeness we close all the neighbourhood functions of the canonical model for the base logic under intersection:
\begin{definition}\label{def:canmodint}
$\mathfrak{M}^c_{\cap} = \langle W^c, \langle \mathcal{N}^{c\cap}_i\rangle_{i\in I}, V^c\rangle$, where for all $(\Lambda,f)\in W^c$, $\mathcal{N}^{c\cap}_i(\Lambda,f)$ is the closure of $\mathcal{N}^c_i(\Lambda,f)$ under (possibly infinite) intersections: $\mathcal{N}^{c\cap}_i(\Lambda,f) = \{\bigcap \mathcal{Y} \mid \mathcal{Y}\subseteq \mathcal{N}^c_i(\Lambda,f)\}$.
\end{definition}

Again, right-to-left of the truth lemma for $\square_G$ is easy, since we only added neighbourhoods to the original canonical model. For left-to-right, we need a slightly different reasoning. Suppose that $\mathfrak{M}^{c\cap}, (\Lambda,f)\models \square_G\varphi$. So there is a $\mathcal{Z} = \{Z_{i}\mid i\in G\}$ such that each $Z_{i} \in\mathcal{N}_i^{c\cap}(\Lambda,f)$, and $\bigcap\mathcal{Z} = \|\varphi\|^{\mathfrak{M}^{c\cap}}$. By the definition of $\mathfrak{M}^{c\cap}$, for every $Z_{i}\in\mathcal{Z}$ there is a $\mathcal{X}_{i} \subseteq_f \mathcal{N}_i(\Lambda,f)$ such that $Z_{i} = \bigcap \mathcal{X}_{i}$. Let $\mathcal{X} = \bigcup_{i\in G} \mathcal{X}_{i}$. Note that $\bigcap\mathcal{X} = \bigcap\mathcal{Z}$. Let $\mathcal{A} = \{(H,\psi)\in\mathbb{G}\times\mathfrak{L}\mid X^{H,\psi}_i\in \mathcal{X} \mbox{ for some } i\in H\}$ and let $\mathcal{B} = \{(H,\psi)\in\mathbb{G}\times\mathfrak{L}\mid X^{H,\psi}_i\in \mathcal{X} \mbox{ for all } i\in H\}$. Note that for all $(H,\psi)\in\mathcal{B}$, $H\subseteq G$.

We now reason as before, deriving the following equation:

\begin{align} \bigcap_{(H,\psi)\in\mathcal{B}} \{(\Lambda,f)\in W^c\mid \psi\in\Lambda\} = \{(\Lambda,f)\in W^c\mid \varphi\in\Lambda\}\end{align}

In other words, every maximal consistent set that contains all $\psi$ for $(H,\psi)\in \mathcal{B}$ also contains $\varphi$, and vice versa. Note however that $\mathcal{B}$ needn't be finite. By the compactness of our syntactic consequence relation however, it follows that there is a finite $\mathcal{C}\subseteq \mathcal{B}$ such that:

\begin{align}
\bigwedge_{(H,\psi)\in\mathcal{C}} \psi \leftrightarrow \varphi
\end{align}

Put $K = \bigcup_{(H,\psi)\in\mathcal{C}} H$. In view of the preceding, $K\subseteq G$. From there, we reason as before, deriving that $\square_K\varphi \in \Lambda$, and finally also that $\square_G\varphi\in\Lambda$.
\end{proof}

\end{document}